\documentclass[12 pt,twosided,reqno]{amsart}
\usepackage{amscd,amssymb}
\usepackage[arrow,matrix]{xy}
\usepackage{graphicx}
\usepackage{epsfig}
\theoremstyle{plain}
\newtheorem{theorem}[subsection]{Theorem}
\newtheorem{lemma}[subsection]{Lemma}
\newtheorem{proposition}[subsection]{Proposition}
\newtheorem{cor}[subsection]{Corollary}

\theoremstyle{definition}
\newtheorem{remark}[subsection]{Remark}
\newtheorem{definition}[subsection]{Definition}
\newtheorem{example}[subsection]{Example}
\newtheorem{examples}[subsection]{Examples}

\newcommand{\PN}{\put(20,0){\line(-1,1){20}}\qbezier(0,0)(4,4)(8,8)\qbezier(20,20)(16,16)(12,12)}

\def\Div{\hbox{Div}}
\def\deg{\hbox{deg}}

\def\Rev{\hbox{Rev}}
\def\supp{\hbox{supp}}
\def\e-supp{\hbox{e-supp}}
\def\inn{\hbox{inn}}

\begin{document}
\title [Simple braids]
{Simple braids }
\thanks{This research is partially supported by Higher Education Commission, Pakistan.\\
2010 AMS classification: Primary 20F36, 57M25; Secondary 57M27, 05A05 .}
\author{  REHANA ASHRAF$^{1}$,\,\,BARBU BERCEANU$^{1,2}$}

\address{$^{1}$Abdus Salam School of Mathematical Sciences,
 GC University, Lahore-Pakistan.}
\email {rashraf@sms.edu.pk}
\address{$^{2}$
 Institute of Mathematics Simion Stoilow, Bucharest-Romania (permanent address).}
\email {Barbu.Berceanu@imar.ro}
  \maketitle
 \pagestyle{myheadings} \markboth{\centerline {\scriptsize
 REHANA ASHRAF,\,\,\,BARBU BERCEANU  }} {\centerline {\scriptsize
 Simple braids}}
\begin{abstract}
We study a subset of square free positive braids and we give a few algebraic characterizations of them and one geometric characterization: the set of positive braids whose closures are unlinks. We describe canonical  forms of these braids and of their conjugacy classes.\end{abstract}
\section{\textbf{Introduction}}
Artin braid group $\mathcal{B}_{n}$ \cite{Artin:47}, the geometrical analogue    of the symmetric group
 $\Sigma_{n}$, is a central object of  study, connected with various mathematical domains. See \cite{Bir:74}, \cite{M},  \cite{Turaev:Kassel}, and also \cite{Paris}  for a recent survey. Garside found a new solution of the word problem and solved the conjugacy problem in $\mathcal{B}_{n}$, using the \emph{braid monoid} $\mathcal{MB}_{n}$ of positive braids \cite{Garside:69}: this is generated by the positive braids $x_{i}$ ($i=1,\ldots,n-1$)

\begin{center}
\begin{picture}(130,50)
\thicklines
\multiput(15,10)(30,0){2}{\line(0,1){20}}
\put(70,10){\PN}
\multiput(115,10)(30,0){2}{\line(0,1){20}}
\put(-10,15){$x_{i}$}
\put(14,35){\tiny $1$}\put(34,35){\tiny$i-1$}\put(67,36){\tiny$i$}\put(83,36){\tiny$i+1$}
\put(110,35){\tiny$i+2$}\put(145,35){\tiny$n$}\multiput(25,17)(100,0){2}{$\cdots$}
\end{picture}
\end{center}
and has Artin defining relations $x_{i}x_{j}=x_{j}x_{i}$ if $|i-j|\neq1$ and $x_{i+1}x_{i}x_{i+1}=x_{i}x_{i+1}x_{i}$.
 The \emph{Garside braid}
$\Delta_{n}=x_{1}(x_{2}x_{1})\ldots (x_{n-1}x_{n-2}\ldots x_{2}x_{1})$ plays a central role: for instance, $\Delta_{n}x_{i}\Delta_{n}^{-1}=x_{n-i}$, and the next four sets of positive braids coincide: divisors of $\Delta_{n}\,(\alpha\,|\,\Delta_{n})$, $\Div(\Delta_{n})=\{\alpha\in \mathcal{MB}_{n}\,|\,\hbox{there exist}\, \delta,\varepsilon\in \mathcal{MB}_{n},\Delta_{n}=\delta\alpha\varepsilon\}$, left divisors of $\Delta_{n},\,(\alpha\,|_{L}\Delta_{n})$, $\Div_{L}(\Delta_{n})=\{\alpha\in \mathcal{MB}_{n}\,|\,\hbox{there exists}\,\varepsilon\in \mathcal{MB}_{n},\Delta_{n}=\alpha\varepsilon\}$,
right divisors of $\Delta_{n}\,(\alpha\,|_{R}\Delta_{n})$ $\Div_{R}(\Delta_{n})=\{\alpha\in \mathcal{MB}_{n}\,|\,\hbox{there exists}\,\delta\in \mathcal{MB}_{n},\Delta_{n}=\delta\alpha\}$, and the set of the square free elements in $\mathcal{MB}_{n}$ ($\alpha \in \mathcal{MB}_{n}$ is square free if there is no generator $x_{i}$ such that $x_{i}^{2}|\alpha$, equivalently if any positive presentation of $\beta$ has no exponent greater than one). Also conjugation  of positive braids in $\mathcal{B}_{n}$ is equivalent with conjugation in $\mathcal{MB}_{n}$ ($\alpha\delta=\delta\beta$ for some positive braid $\delta$) and this can be reduced to a sequence of conjugation with $\delta$ in $\Div(\Delta_{n})$ (see~\cite{Garside:69},~\cite{Bir:74}).

Computing polynomial invariants (Alexander-Conway, Jones, and also $D$) of closed braids we found Fibonacci type recurrences which reduce computations to a new class of square free positive braids (see~\cite{Barbu-Nizami:09},~\cite{Barbu:Rehana}). First we define five sets of positive braids: the set $\mathcal{LSB}_{n}$
 of literally simple braids, the set $\mathcal{CSB}_{n}$ of conjugate simple braids, the invariant simple set $\mathcal{ISB}_{n}$, the set $\mathcal{MSB}_{n}$ of Markov simple braids, the set of $\mathcal{GSB}_{n}$ of geometrically simple braids.

 \begin{definition}Let $\mathcal{MF}_{n-1}$ be the free monoid generated by $x_{1},x_{2},\ldots,x_{n-1}$. An element
$\omega \in \mathcal{MF}_{n-1},\,\omega=x_{i_{1}}x_{i_{2}}\ldots x_{i_{k}}$ is called a \emph{simple word} if $i_{a}\neq i_{b}$ for  $a\neq b$. A positive braid $\alpha \in \mathcal{MB}_{n}$ is called a \emph{literally simple braid} if under the natural projection $\pi :\mathcal{MF}_{n-1}\longrightarrow \mathcal{MB}_{n} $ there exists a simple word $\omega$ such that $\pi(\omega)=\alpha.$
\end{definition}

\begin{definition}
\label{conj:sq free} A positive braid $\beta$ is said to be a
\emph{conjugate simple braid} if all positive braids $\beta'$
conjugate to $\beta$ are square free.
\end{definition}

\begin{examples} 1) $x_{2}x_{1}x_{3}x_{2}x_{1}x_{3}$ is a square free word and also a square free  braid.\\
\indent 2) $x_{3}x_{2}x_{1}x_{3}x_{2}x_{1}$ is a square free word but not a square free braid   because \\ $x_{3}x_{2}x_{1}x_{3}x_{2}x_{1}=x_{2}x_{1}x_{3}x_{2}x_{1}^{2}$.\\
\indent 3) $\beta=x_{1}x_{2}x_{1}$ is a square free braid (it has only two positive
presentations:  $x_{1}x_{2}x_{1}$ and $x_{2}x_{1}x_{2}$), but is neither a simple braid nor a
conjugate simple braid (because $\beta\thicksim x_{1}^{2}x_{2}).$
\end{examples}
We say that a subset $A\subset\mathcal{MB}_{n}$ is \emph{invariant under conjugation} if $(\bigcup\limits_{\alpha\in\mathcal{B}_{n}}\alpha A \alpha^{-1})$\\
$\bigcap\mathcal{MB}_{n} \subset A$. For instance, in $\mathcal{MB}_{n},\,A=\{1,x_{1},\ldots,x_{n-1}\}$ is invariant under conjugation but $B=\{1,x_{1},\ldots,x_{n-2}\}$ is not.

\begin{definition}The \emph{invariant simple set} is the largest subset of $\Div(\Delta_{n})$ invariant under conjugation: $\mathcal{ISB}_{n}=\bigcup\{A\subset \Div(\Delta_{n})\,|\,A\,\hbox{is invariant under conjugation}\}.$
\end{definition}

\begin{definition}
\label{markov:sqfree}
 A positive braid is said to be \emph{Markov simple braid} if any positive
 braid $\beta'$ obtained from $\beta$ by a finite sequence of positive braids $\beta=\beta_{1},\beta_{2},\ldots,\beta_{s}=\beta'$ of moves
 $MI$ and $MII_{+}$ is square free.  Here $MI$ and $MII_{+}$ are classical Markov moves (see \cite{Bir:74}):\\
 \indent $MI: \beta_{i}\rightarrow
 \beta_{i+1},$ where the two braids are conjugate in the same
 $\mathcal{B}_{n};\\
 \indent  MII_{+}: \beta_{i} \in \mathcal{B}_{n+1}, \beta_{i+1} \in
  \mathcal{B}_{n}$ and $\beta_{i}=\beta_{i+1}x_{n}$ or $\beta_{i}\in\mathcal{B}_{n},\,\beta_{i+1}\in\mathcal{B}_{n+1}$
and $\beta_{i+1}=\beta_{i}x_{n}.$
\end{definition}
\noindent The last definition is geometrical, we are looking at the positive part of the "kernel" of the closure map  $C:\coprod\limits_{n}\mathcal{MB}_{n}\longrightarrow Links$:
\begin{definition}A positive braid $\beta$ is said to be a \emph{geometrically simple braid} if its closure $\widehat{\beta}$ is a trivial link.\end{definition}
Each of these sets are studied  in separate sections. Our aim is  to show that all  these notions coincide:
\begin{theorem}\label{main: th}$\mathcal{LSB}_{n}=\mathcal{CSB}_{n}=\mathcal{ISB}_{n}=\mathcal{MSB}_{n}=\mathcal{GSB}_{n}.$
\end{theorem}

Now we call \emph{simple braids} elements of this unique set $\mathcal{SB}_{n}$. We also consider the group  $\mathcal{SB}_{\infty}=\bigcup\limits_{n\geq 1}\mathcal{SB}_{n}$ and the set $\mathcal{SB}=\coprod\limits_{n\geq 1}\mathcal{SB}_{n}$ (for Markov moves  and closure of braids, it is necessary to know the number of strands of a braid).\\
\indent We will also give  canonical forms for simple braids and their conjugacy classes. Here "canonical forms" of $\beta$ has a precise meaning: in the set of words in the free monoid $\mathcal{MF}_{n-1}$ representing the element $\beta\in\mathcal{MB}_{n}$, this is called the diagram of $\beta$ in \cite{Garside:69}, \cite{Bir:74}, we always choose the minimal one in the length-lexicographic order  given by $x_{1}<x_{2}<\ldots <x_{n-1}$, and similarly for the set of words representing a conjugacy class in $\mathcal{MB}_{n}$. For instance, the canonical form of divisor of $\Delta_{n}$ is given by
$$\beta_{K,J}=\beta_{k_{1},j_{1}}\beta_{k_{2},j_{2}}\ldots
\beta_{k_{s},j_{s}}\,\,\,\,\,\,\,\,\,\,\,\,\,\,\,(*)$$
where  $\beta_{k,j}=x_{k}x_{{k}-1}\ldots x_{j}\,\,(j\leq k)$, the sequence  $K=(k_{1},\ldots,k_{s})$ is increasing, and the sequence $J=(j_{1},j_{2},\ldots,j_{s})$ satisfies $j_{h}\leq k_{h}$ ($h=1,\ldots,s$); this is a consequence of the form of G\"{o}bner basis for $\mathcal{MB}_{n}$,  see~\cite{Boku},~\cite{Zaffar},~\cite{Usman},~\cite{Barbu:usman} for related results and~\cite{Barbu},~\cite{Barbu:usman} for more details and the proof of $(*)$.

We have a decomposition theorem, similar to the decomposition of permutations (see section 2 for definitions of braid cycles, disjoint cycles, and their partial order):
\begin{theorem}\label{|unique representation of cycle} Every simple braid $\alpha \in \mathcal{SB}_{n} $ can be written in a unique way as a product of  disjoint cycles  $\alpha=\gamma_{1}\gamma_{2}\ldots\gamma _{r}$, where $\gamma_{1}\prec\gamma_{2}\prec\ldots \prec\gamma_{r}$.
\end{theorem}

For the conjugacy classes of elements in $\mathcal{SB}_{n}$ or $\mathcal{SB}_{\infty}$ we have

\begin{theorem}\textbf{(Canonical form of conjugacy class)}\label{canonical:conj sq fre} a) A  simple braid $\beta\in \mathcal{SB}_{\infty}$ is conjugate to the braid
 $$\beta_{A}=(x_{1}x_{2}\ldots x_{s_{1}-1})(x_{s_{1}+1}\ldots x_{s_{2}-1})\ldots (x_{s_{r-1}+1}\ldots x_{s_{r}-1})$$ where $A=(a_{1},a_{2},\ldots,a_{r})$
is a sequence of integers satisfying
$a_{1}\geq a_{2}\geq\ldots \geq a_{r}\geq 2$ and $s_{i}=a_{1}+a_{2}+\ldots +a_{i}.$\\
b) If $\beta_{A}\sim \beta_{A'}$ where $\beta_{A}$ and $\beta_{A'}$
are as in part a), then $A=A'$.
\end{theorem}
Here is a picture of a simple braid in $\mathcal{B}_{13}:$ if $A=(4,3,2,2)$ then the corresponding braid  $\beta_{A}$ is $(x_{1}x_{2}x_{3})(x_{5}x_{6})(x_{8})(x_{10})$.

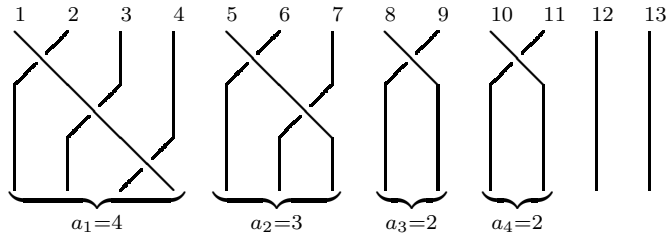
\begin{center}
\begin{picture}(300,90)
\thicklines
\put(10,10){\line(0,1){40}}\put(30,10){\line(0,1){20}}\put(50,50){\line(0,1){20}}\put(70,30){\line(0,1){40}}
\put(10,50){\PN}\put(30,30){\PN}\put(50,10){\PN}
\put(90,10){\line(0,1){40}}\put(110,10){\line(0,1){20}}\put(130,10){\line(0,1){20}}\put(130,50){\line(0,1){20}}
\put(90,50){\PN}\put(110,30){\PN}
\multiput(150,10)(20,0){4}{\line(0,1){40}}\multiput(150,50)(40,0){2}{\PN}
\multiput(230,10)(20,0){2}{\line(0,1){60}}
  \put(8,10){{$\underbrace{\,\,\,\,\,\,\,\,\,\,\,\,\,\,\,\,\,\,\,\,\,\,\,\,\,\,\,\,\,\,\,\,\,}_{a_{1}=4}$}}
   \put(85,10){{$\underbrace{\,\,\,\,\,\,\,\,\,\,\,\,\,\,\,\,\,\,\,\,\,\,\,\,}_{a_{2}=3}$}}
   \put(147,10){{$\underbrace{\,\,\,\,\,\,\,\,\,\,\,\,\,}_{a_{3}=2}$}}
   \put(187,10){{$\underbrace{\,\,\,\,\,\,\,\,\,\,\,\,\,}_{a_{4}=2}$}}
   \put(10,74){\tiny$1$}\put(30,74){\tiny$2$}\put(50,74){\tiny$3$}\put(70,74){\tiny$4$}\put(90,74){\tiny$5$}\put(110,74){\tiny$6$}
   \put(130,74){\tiny$7$}\put(150,74){\tiny$8$}\put(170,74){\tiny$9$}\put(190,74){\tiny$10$}\put(210,74){\tiny$11$}
   \put(228,74){\tiny$12$}\put(248,74){\tiny$13$}
\end{picture}
\end{center}
\medskip
\begin{cor}A simple braid $\beta\in \mathcal{SB}=\coprod\limits_{n\geq 1}\mathcal{SB}_{n}$ is Markov equivalent with $1_{n}$, the unit braid in some $\mathcal{MB}_{n}$.
\end{cor}
The canonical projection of the braid group  to the symmetric group $\Sigma_{n}$ restricted to the square free braids gives a bijection; restricted to the simple braids gives a bijection between conjugacy classes (for a subset $A\subset \mathcal{B}_{n}$, $A/_{\sim}$ denote the set of conjugacy classes intersecting $A$):

\begin{cor}\label{cor:diagram}There is a commutative diagram of sets where $s$ and $\pi'$ are bijections:
\begin{center}
\begin{picture}(300,130)
\thicklines
\put(10,10){$\mathcal{SB}_{n}/_{\large\sim}$}
\put(50,13){\vector(1,0){210}}
\put(30,25){\vector(2,1){30}}
\put(60,45){$\Div(\Delta_{n})/_{\sim}$}
\put(117,50){\vector(1,0){40}}
\put(160,45){$\mathcal{MB}_{n}/_{\sim}$}
\put(190,90){\vector(1,0){40}} \put(202,50){\vector(1,0){25}}
\put(230,45){$\mathcal{B}_{n}/_{\sim}$}
\put(243,38){\vector(2,-1){30}}\put(262,10){$\Sigma_{n}/_{\sim}$}
\put(163,85){$\mathcal{MB}_{n}$}\put(235,85){$\mathcal{B}_{n}$}
\multiput(175,80)(67,0){2}{\vector(0,-1){20}}
\put(245,90){\vector(1,1){23}}
\put(270,115){$\Sigma_{n}$}
\put(280,110){\vector(0,-1){85}}\put(265,120){\vector(-1,0){155}}
\put(60,115){$\Div(\Delta_{n})$}
\put(75,108){\vector(0,-1){50}}\put(100,105){\vector(4,-1){60}}
\put(10,115){$\mathcal{SB}_{n}$}
\put(19,110){\vector(0,-1){85}}\put(30,120){\vector(1,0){25}}
\put(150,18){$\pi'$}\put(148,2){$\approx$}\put(180,123){$s$}\put(178,110){$\approx$}
\put(260,95){$\pi$}
\end{picture}
\end{center}
\end{cor}
Familiarity with Garside paper~\cite{Garside:69}, the canonical form of square free braids $(*)$, and simple properties of the polynomial invariant for links $D$, a new specialization of HOMFLY polynomial (see~\cite{Barbu:Rehana}), make the paper self contained. Elementary combinatorics of simple braids will be discussed in~\cite{Barbu:Rehana3}. We hope the reader will enjoy finding new properties of simple braids, new applications, and also shorter proofs of these results.

\section{\textbf{Literally simple braids}}
First remark that the definition of a literally simple braid does not depend on the representative: if $\alpha =\pi(\omega)=\pi(\omega')$, where $\omega,\omega'\in\mathcal{MF}_{n-1}$  and $\omega$ is a simple word, then $\omega'$ is also a simple word (only commutation relations can be used). It is obvious that $\mathcal{LSB}_{n}$ satisfies the following properties:
\begin{proposition}
\indent 1) $\mathcal{LSB}_{n}\subset \Div(\Delta_{n})$.\\
\indent 2) If $\alpha\in \mathcal{LSB}_{n}$ and $\beta\,|\,\alpha$,  then $\beta\in\mathcal{LSB}_{n}$.\\
\indent 3) $\mathcal{LSB}_{n}$ is invariant under Garside involutions:\\
\indent \indent 3.1) $\Delta_{n}\mathcal{LSB}_{n}\Delta_{n}^{-1}=\mathcal{LSB}_{n}$\\
\indent \indent 3.2) $\Rev(\mathcal{LSB}_{n})=\mathcal{LSB}_{n}$.
\end{proposition}
\noindent Here $\Rev(x_{i_{1}}\ldots x_{i_{s}})=x_{i_{s}}\ldots x_{i_{1}}$, see~\cite{Garside:69}.
\begin{example}We will use two types of (very) simple braids: $U(a,b)=x_{a}x_{a+1}\ldots x_{b}$, where $1\leq a\leq b\leq n-1$, and $D(c,d)=x_{c}x_{c-1}\ldots x_{d+1}x_{d}$, where $1\leq d<c\leq n-1$; for instance, $x_{3}$ is $U(3,3)$ but not $D(3,3)$.
\end{example}
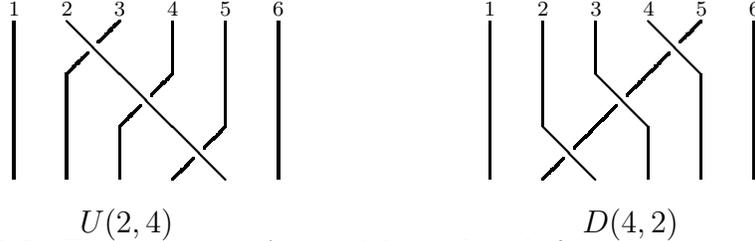
\begin{center}
\begin{picture}(300,80)
\thicklines
\put(10,10){\line(0,1){60}}
\put(30,10){\line(0,1){40}}
\put(50,10){\line(0,1){20}}
\put(70,50){\line(0,1){20}}
\put(90,30){\line(0,1){40}}
\put(110,10){\line(0,1){60}}
\multiput(30,50)(10,0){1}{\PN}
\multiput(50,30)(10,0){1}{\PN}
\multiput(70,10)(10,0){1}{\PN}
\put(190,10){\line(0,1){60}}
\put(210,30){\line(0,1){40}}
\put(230,50){\line(0,1){20}}
\put(250,10){\line(0,1){20}}
\put(270,10){\line(0,1){40}}
\put(290,10){\line(0,1){60}}
\multiput(210,10)(10,0){1}{\PN}
\multiput(230,30)(10,0){1}{\PN}
\multiput(250,50)(10,0){1}{\PN}
\put(35,-10){$U(2,4)$}
\put(225,-10){$D(4,2)$}
\put(8,72){\tiny$1$}\put(28,72){\tiny$2$}\put(48,72){\tiny$3$}\put(68,72){\tiny$4$}\put(88,72){\tiny$5$}\put(108,72){\tiny$6$}
\put(188,72){\tiny$1$}\put(208,72){\tiny$2$}\put(228,72){\tiny$3$}\put(248,72){\tiny$4$}\put(268,72){\tiny$5$}\put(288,72){\tiny$6$}
\end{picture}
\end{center}

\begin{definition} The support of a positive $n$-braid  $\beta$ is the set $\supp(\beta)=\{i\in\{1,\ldots,n-1\}\,|\,x_{i}\in\Div(\beta)\}$. The support of $\beta$ is \emph{connected} if it is an integral interval $[a,b]$. The supports of $\alpha$  and $\beta$ are \emph{consecutive} if  $\supp(\alpha)=[a,b]$  and  $\supp(\beta)=[b+1,c]$. In the case of connected support the \emph{extended support} is $\e-supp(\beta)=[\max(1,a-1),\min(b+1,n-1)]$. For instance, the braid $\beta\in \mathcal{MB}_{11},\,\beta=U(3,5)D(8,6)U(9,10)$ has a connected support $\supp (\beta)=[3,10]$ and $\e-supp(\beta)=[2,10]$.

%The interval  $[a,b]$ is called \emph{support} of $U(a,b)$, \sup$(U(a,b))$ and the interval $[\hbox{max}\{1,a-1\},\hbox{min}\{b+1,n-1\}]$ is called \emph{extended support} of $U(a,b)$, \e-sup$(U(a,b))$.  Similarly the interval $[d,c]$ is called \emph{support} of $D(c,d)$, \sup$(D(c,d))$ and the interval $[\hbox{max}\{1,d-1\},\hbox{min}\{c+1,n-1\}]$ is called \emph{extended support} of $D(c,d)$, \e-sup$(U(a,b))$.
\end{definition}

\begin{definition}A \emph{cycle} $\gamma$ is a literally simple braid, product of factors $U$ and $D$ with consecutive supports but not two consecutive factors $U$; the unit braid $1$ is not a cycle.
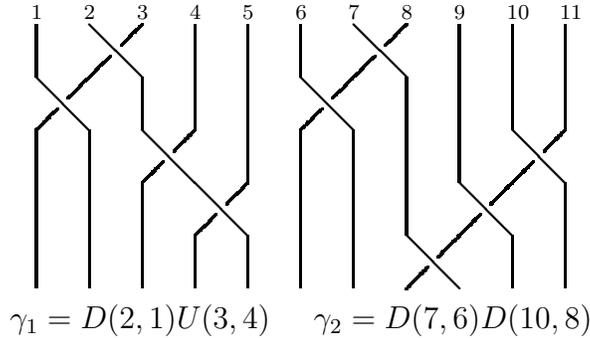
\begin{center}
\begin{picture}(230,120)
\thicklines
\put(10,10){\line(0,1){60}}
\put(10,90){\line(0,1){20}}
\put(30,10){\line(0,1){60}}
\put(50,10){\line(0,1){40}}
\put(50,70){\line(0,1){20}}
\put(70,10){\line(0,1){20}}
\put(70,70){\line(0,1){40}}
\put(90,10){\line(0,1){20}}
\put(90,50){\line(0,1){60}}
\put(110,10){\line(0,1){60}}
\put(110,90){\line(0,1){20}}
\put(130,10){\line(0,1){60}}
\put(150,30){\line(0,1){60}}
\put(170,50){\line(0,1){60}}
\put(190,10){\line(0,1){20}}
\put(190,70){\line(0,1){40}}
\put(210,10){\line(0,1){40}}
\put(210,70){\line(0,1){40}}
\multiput(10,70)(10,0){1}{\PN}
\multiput(30,90)(10,0){1}{\PN}
\multiput(50,50)(10,0){1}{\PN}
\multiput(70,30)(10,0){1}{\PN}
\multiput(110,70)(10,0){1}{\PN}
\multiput(130,90)(10,0){1}{\PN}
\multiput(150,10)(10,0){1}{\PN}
\multiput(170,30)(10,0){1}{\PN}
\multiput(190,50)(10,0){1}{\PN}

\put(0,-5){$\gamma_{1}=D(2,1)U(3,4)$}

\put(115,-5){$\gamma_{2}=D(7,6)D(10,8)$}
\put(8,112){\tiny$1$}\put(28,112){\tiny$2$}\put(48,112){\tiny$3$}\put(68,112){\tiny$4$}\put(88,112){\tiny$5$}\put(108,112){\tiny$6$}
\put(128,112){\tiny$7$}\put(148,112){\tiny$8$}\put(168,112){\tiny$9$}\put(188,112){\tiny$10$}\put(208,112){\tiny$11$}
\end{picture}
\end{center}
\end{definition}

\begin{remark} 1) Factorization of a cycle $\gamma$ as a product $\ldots (D\ldots D)U(D\ldots D)U\ldots$ with consecutive supports is unique ($U(a,b)U(b+1,c)$ should be replaced by $U(a,c)$).\\
\indent 2) If $\gamma$ is a cycle then $\supp(\gamma)=(\bigcup \supp(U_{i}))\bigcup(\bigcup\supp D_{j})$ is connected (and the union is a disjoint union).\\

\end{remark}

\begin{definition}For two cycles $\gamma_{1}$ and $\gamma_{2}$ with $\sup(\gamma_{1})=[a,b],\,\sup(\gamma_{2})=[c,d]$ we define a partial order by $\gamma_{1} \prec \gamma_{2}$, if $c\geq b+1$. If $c=b+1$, $\gamma_{1}$ and $\gamma_{2}$ are \emph{consecutive}. For $c\geq b+2$, $\gamma_{1}$ and $\gamma_{2}$ are called \emph{disjoint} cycles and  for $c\geq b+3$, $\gamma_{1}$ and $\gamma_{2}$  are called \emph{distant} cycles. We will extend this partial order to literally simple braids: if $\alpha=\gamma_{1}\ldots \gamma_{a},\,\beta=\gamma'_{1}\ldots \gamma'_{b}$ we define $\alpha\prec\beta$ if $\gamma_{1}\prec\gamma_{2}\prec,\ldots,\prec\gamma_{a}\prec\gamma'_{1}\prec\ldots \prec\gamma'_{b}$. In this case we say that $\alpha$ and $\beta$ are \emph{disjoint} (\emph{distant}) simple braids if $\gamma_{a}$ and $\gamma'_{1}$ are disjoint (distant) cycles.
\end{definition}
\begin{proposition}\label{| prop:unique representation of cycle} Every literally simple braid $\alpha \in \mathcal{LSB}_{n} $ can be written in a unique way as product of an increasing sequence of disjoint cycles  $\alpha=\gamma_{1}\gamma_{2}\ldots\gamma _{r}$.
\end{proposition}
\begin{proof}
If the square free braid $\beta_{K,J}=\beta_{k_{1},j_{1}}\beta_{k_{2},j_{2}}\ldots \beta_{k_{s},j_{s}}$ is literally simple, we have $j_{h+1}>k_{h}$ for $1,2,\ldots ,s-1$ (no condition if $s\leq1$). Replace $\beta_{k,j}$ by $D(k,j)$ if $k>j$ and by $U(k,j)$ if $k=j$, next recollect products $U(k,k)U(k+k,k+1)\ldots U(l,l)$ into $U(k,l)$ and multiplying factors $U's$ and $D's$
with consecutive supports, find the product of disjoint factors $\gamma_{1}\prec\gamma_{2}\prec\ldots \prec\gamma_{r}$ (the number $r$ of cycles is at most the number $s$ of $\beta_{k,j}$ factors). Factorization  is unique because the support $\supp(\beta_{K,J})$ has the decomposition in connected components (and increasing order) the disjoint union $\coprod^{r}\limits _{i=1}\supp(\gamma_{i})$.
\end{proof}
\begin{remark}When the proof of the Theorem \ref{main: th} will be completed, the above proof will be a proof of Theorem \ref{|unique representation of cycle}.
\end{remark}

\begin{definition} (\cite{Elrifai: Morton}) If  $\beta \in \mathcal{MB}_{n}$, we denote by  $\inn(\beta)$ the \emph{initial set} of $\beta$: $\{i\,|\,x_{i}\in\Div_{L}(\beta)\}$.
\end{definition}
\begin{proposition}a) If $\gamma$ is a cycle with canonical factorization $\gamma=\ldots (D\ldots D)U$\\$(D\ldots D)U \ldots$, then $\inn(\gamma)=\{\hbox{~the index of the first letter of the first factor of }\gamma $ and the indices of the first letters of $D$ factors of $\gamma\}$.\\
\indent b) If $\alpha$ is a simple braid written in canonical form  $\alpha=\gamma_{1}\prec\gamma_{2}\prec\ldots\prec\gamma_{s}$ with disjoint cycles $\gamma_{i}$, then $\inn(\alpha)=\coprod^{s}\limits_{i=1}\inn(\gamma_{i})$.
%$=\{\hbox{the indices of the first letters of $\gamma_{i}^{'}s$ and $D^{'}s$ factors}\}$.
\end{proposition}
\begin{proof}a) If $\gamma=x_{b}\ldots D(a,c)\ldots$ then obviously $x_{b}|_{L}\gamma$ and $x_{a}$ commutes with all the factors before $x_{a}$ (the factors $x_{a-1}$ should be in $D(a,c)$ and $x_{a+1}$ could appear only in the factors after $x_{a}$), hence $x_{a}|_{L}\gamma$. For opposite inclusion we will use the divisibility properties from section 7:  if $x_{k}|_{L}\gamma$ then $k\in\supp(\gamma)$ and we have two cases: $x_{k}$ is a divisor of a $U$ factor or of a $D$ factor. In the first case, $x_{k}|U$ (and $k\geq b+1$), the factor $x_{k-1}$ appears before $x_{k}$: $\gamma=F_{1}\cdot x_{k-1}\cdot F_{2}\cdot x_{k}\cdot F_{3}$; Garside Lemma \ref{garside} and Proposition \ref{leter:F} imply $x_{k}x_{k-1}|_{L}F_{2}\cdot x_{k}\cdot F_{3}$ but this is impossible because $k-1\notin \supp(F_{2}\cdot x_{k}\cdot F_{3})$. In the second case, $x_{k}|D(a,c)$ (and $k\leq a-1$), the factor $x_{k+1}$ is in front of $x_{k}$: $\gamma=F_{1}\cdot x_{k+1}x_{k}\cdot F_{2}$. Lemma \ref{garside} and Proposition \ref{leter:F} imply $x_{k}|_{L}x_{k+1}x_{k}\cdot F_{2}$, hence $x_{k+1}|_{L}F_{2}$ and this is not possible because $x_{k+1}\notin\supp(F_{2})$.\\
\indent b) Because any $x_{k}$, $k\in\supp(\gamma_{i})$, commutes with all $\gamma_{j}$, $j\neq i$, the formula is obvious.
\end{proof}
\begin{example} If $\alpha=\gamma_{1}\gamma_{2}$, $\gamma_{1}=U(2,4)D(6,5)D(9,7)U(10,11)$, $\gamma_{2}=U(13,14)$\\$D(17,15)$, then $\inn(\alpha)=\{2,6,9,13,17\}$. For the computation of $\inn(\gamma)$ for a positive braid $\alpha$, see \cite{Barbu:usman}.
\end{example}
\section{\textbf{Conjugate simple  braids}}
We start to show that $\mathcal{CSB}_{n}\subseteq \mathcal{LSB}_{n}$:
\begin{lemma}
\label{crit:conj sq fre} If
$\beta_{K,J}=\beta_{k_{1},j_{1}}\beta_{k_{2},j_{2}}\ldots
\beta_{k_{s},j_{s}}$  (where $1\leq k_{1}<k_{2}<\ldots k_{s}\leq
n-1$ and $j_{i}\leq k_{i}$ for all $i=1,2,\ldots,s$) is a conjugate simple braid in $\mathcal{B}_{n}$, then  $j_{i+1}>k_{i}$ for all
$i=1,2,\ldots,s-1$.
\end{lemma}
\begin{proof} The proof is by double induction on the number $s$ of blocks
$\beta_{k,j}$ and on the length of the last block
$\beta_{k_{s},j_{s}}$. During this proof we conjugate a positive
braid $\beta$ with positive braids $\gamma$ involving only letters with indices in $\supp(\beta)$.  Given a braid violating the condition
$j_{i+1}> k_{i}$ for some $i$, we conjugate this braid to obtain another
one containing a square or having a smaller number of blocks or a
smaller length of the last block and still containing a pair
$j_{h+1}\leq k_{h}.$
The induction starts with $s=1$ (one block) or $k_{s}=j_{s}>k_{s-1}$ (the last block is a singleton).

Let us analyze  the case where $k_{s-1}<j_{s}(\leq k_{s}).$

\emph {Case 1:} $j_{s}>k_{s-1}+1$. The first $s-1$ blocks contains a
pair $j_{i+1}>k_{i}$ and there is a conjugate of this braid (using
only letters which commutes with $\beta_{k_{s},j_{s}}$) containing
squares (induction on $s$).

\emph {Case 2:} $j_{s}=k_{s-1}+1$. In this case $\beta_{k_{s},j_{s}}$
commutes with $\beta_{k_{i},j_{i}}$  for $i=1,\ldots,s-2:$
$$\beta_{K,J}\sim\beta_{k_{s},j_{s}}\beta_{k,j}\beta^{-1}_{k_{s},j_{s}}=\beta_{k_{1},j_{1}}\beta_{k_{2},j_{2}}\ldots
(\beta_{k_{s},j_{s}}\beta_{k_{s-1},j_{s-1}}).$$ and we reduced the
number of
blocks by one, and again we have a pair $j_{i+1}\geq k_{i}.$\\
\indent Now we start the analyze the case  $j_{s}\leq k_{s-1}.$
Conjugate $\beta_{K,J}$ with $x_{j_{s}}$, and denote
$\beta_{K,J}\sim \beta'=x_{j_{s}}\beta_{k,j}x^{-1}_{j_{s}}.$

\emph {Case 3:} $j_{s}<k_{s-1}$. We divide the computation of  $\beta'=x_{j_{s}}\beta_{K,J}x_{j_{s}}^{-1}$ into four subcases:

3.1) there exists $k_{a}=j_{s}-1,\,k_{a+1}=j_{s}$, then
\begin{eqnarray*}
  \beta' &=& \beta_{k_{1},j_{1}}\beta_{k_{2},j_{2}}\ldots
x_{j_{s}}(x_{k_{a}}\ldots x_{j_{a}})(x_{k_{a+1}}\ldots x_{j_{a+1}})
\ldots\beta_{k_{s},j_{s}+1} \\
   &=& \beta_{k_{1},j_{1}}\beta_{k_{2},j_{2}}\ldots
(x_{k_{a}+1}x_{k_{a}}\ldots x_{j_{a}})(x_{k_{a+1}}\ldots
x_{j_{a+1}}) \ldots\beta_{k_{s},j_{s}+1}
\end{eqnarray*}
and we have two subcases:

3.1.1) if $j_{a}\leq j_{a+1}$, then
\begin{eqnarray*}
  \beta' &=& \beta_{k_{1},j_{1}}\beta_{k_{2},j_{2}}\ldots
(x_{k_{a+1}}x_{k_{a}}\ldots x_{j_{a+1}}\ldots
x_{j_{a}})(x_{k_{a+1}}\ldots x_{j_{a+1}})
\ldots\beta_{k_{s},j_{s}+1} \\
   &=&  \beta_{k_{1},j_{1}}\beta_{k_{2},j_{2}}\ldots
(x_{k_{a}}\ldots x_{j_{a+1}-1})(x_{k_{a+1}}\ldots x_{j_{a}})
\ldots\beta_{k_{s},j_{s}+1}
\end{eqnarray*}
and in this canonical form of  $\beta'_{K,J}$ the last index in the
$a+1$-block is not greater than the first index in the $a$-block:
$x_{k_{a}}\geq
x_{j_{a}};$

3.1.2) otherwise  $j_{a}>j_{a+1}$, and
\begin{eqnarray*}
    \beta' &=& \beta_{k_{1},j_{1}}\beta_{k_{2},j_{2}}\ldots
(x_{k_{a+1}}x_{k_{a}}\ldots x_{j_{a}})(x_{k_{a+1}}\ldots
x_{j_{a}}\ldots  x_{j_{a+1}}) \ldots\beta_{k_{s},j_{s}+1}\\
&=&\beta_{k_{1},j_{1}}\beta_{k_{2},j_{2}}\ldots (x_{k_{a}}\ldots
x_{j_{a}})(x_{k_{a+1}}\ldots x^{2}_{j_{a}}\ldots x_{j_{a+1}})
\ldots\beta_{k_{s},j_{s}+1}
\end{eqnarray*}
which contains $x^{2}_{j_{a}};$

3.2) there exists an index $k_{a}=j_{s}$, but none is equal to
$j_{s}-1$: now
\begin{eqnarray*}
   \beta'&=& \beta_{k_{1},j_{1}}\beta_{k_{2},j_{2}}\ldots
x_{j_{s}}(x_{k_{a}}\ldots x_{j_{a}}) \ldots\beta_{k_{s},j_{s}+1}
\end{eqnarray*}
which contains $x^{2}_{j_{s}};$

3.3) there exists an index $k_{a}=j_{s}-1$, but none is equal
to $j_{s}$, then
\begin{eqnarray*}
     \beta' &=& \beta_{k_{1},j_{1}}\beta_{k_{2},j_{2}}\ldots
x_{j_{s}}(x_{k_{a}}\ldots x_{j_{a}}) \ldots\beta_{k_{s},j_{s}+1}\\
&=& \beta_{k_{1},j_{1}}\beta_{k_{2},j_{2}}\ldots
(x_{k_{a+1}}x_{k_{a}}\ldots x_{j_{a}}) \ldots\beta_{k_{s},j_{s}+1}
\end{eqnarray*}
and the last index of the last block is too small: $j_{s+1}\leq k_{s-1}$ and also the length of the last block is smaller;

3.4) there does not exist any $k_{a}=j_{s}$ or $j_{s}-1$: after a permutation with the first blocks,
\begin{eqnarray*}
  \beta'&=&\beta_{k_{1},j_{1}}\ldots \beta_{k_{h},j_{h}}(x_{j_{s}})\beta_{k_{h+1},j_{h+1}}\ldots \beta
  _{k_{s},j_{s+1}}
\end{eqnarray*}
and we repeat the previous argument.

\emph {Case 4:} $j_{s}=k_{s-1}.$ This is divided into two subcases:

4.1) for $k_{s-2}=j_{s}-1$, apply the Case 3.1) for $a=s-2$;

4.2) for  $k_{s-2}\neq j_{s}-1,$ use the Case 3.2) for $a=s-1.$
\end{proof}
\begin{cor} $\mathcal{CSB}_{n}\subseteq \mathcal{LSB}_{n}$.
\end{cor}
The proof of the opposite inclusion is longer; the key steps are the next gud and baf   Lemmas.
\begin{lemma}(\textbf{Going up and down})\label{up n down} If  $\gamma$ is a cycle, $\beta,\delta\in \mathcal{MB}_{n}$ and $\gamma\beta=\beta\delta$, then $\inn(\beta)\bigcap\supp(\gamma)\neq \emptyset$ implies that
$\inn(\beta)\bigcap\inn(\gamma)\neq\emptyset$.
\end{lemma}
\begin{proof}
\emph{Going down case:} Suppose that $\gamma$ has a factor $U(a,b)$ and there is an index $i\in \inn(\beta)\bigcap[a,b]$. First we want to show that $a\in\inn(\beta)$ by induction: if $a<i,i\in\inn(\beta)$, then $i-1\in\inn(\beta)$. All the factors before $U(a,b)$ (if any) commute with $x_{i}$ and Lemma \ref{garside} a) implies that $x_{i}|_{L}U(a,b)D(c,b+1)\ldots \beta$; by Proposition \ref{l.c.m ,U} d),  (in the case when $U(a,b)$ is the last factor of $\gamma$ we obtain directly $x_{i-1}|_{L}\beta$). Using again Lemma \ref{garside} a) ($i-1\leq b-1$), we obtain $x_{i-1}|_{L}\beta$. Now suppose that $a\in\inn(\beta)$ and $\gamma=\ldots U(a,b)\ldots $. If $U(a,b)$ is the first factor of $\gamma$, then $a\in \inn(\gamma)$, otherwise $\gamma=\ldots D(a-1,d)U(a,b)\ldots$ with $d\leq a-2$. $x_{i}$ commutes with all the factors before $D(a-1,d)$ (if there are such factors), therefore Lemma \ref{garside} a) implies  $x_{a}|_{L}D(a-1,d)U(a,b)\ldots \beta$. Proposition \ref{l.c.m,D} e) implies $D(a-1,d)D(c,d)U(a,b)\ldots \beta$, hence $x_{a-1}|_{L}U(a+1,b)\ldots \beta$, and again Lemma \ref{garside} a) gives $x_{a-1}|_{L}\beta$, and this element is in $\inn(\beta)\bigcap\inn(\gamma)$.\\
%%%%%%%%%%%%%%%%%%%%%%%%%%%%%%%%%%%%%%%%%%%%%%%%%%%%%%%%%
\indent \emph{Going up case:} Suppose now that $\gamma$ has a factor $D(a,b)\,(a\geq b+1)$ and there is an index $i\in\inn(\beta)\bigcap[b,a]$. We want to show that $a\in\inn(\beta)$. If $i\in[b,a-1]$, we show that $i+1\in\inn(\beta)$ and by induction we obtain the result. We start with the simplest case $i\in[b+1,a-1]$. All the factors before $D(a,b)$ (if any) commute with $x_{i}$ and Lemma \ref{garside} a) implies that $x_{i}|_{L}D(a,b)\ldots \beta$, hence by Proposition \ref{l.c.m,D} c), we obtain that $D(a,b)x_{i+1}$ divides $D(a,b)\beta$ (if $D(a,b)$ is the last factor of $\gamma$) or divides $D(a,b)U(a+1,c)\ldots\beta$ or divides $D(a,b)D(c,a+1)\ldots\beta$. In the first case we have $x_{i}|_{L}\beta$. In the last two cases, if $i+1\leq a-1$, $x_{i+1}$ commutes with the last factors of $\gamma$ and Lemma \ref{garside} a) implies  $x_{i+1}|_{L}\beta$; if $i+1=a$, we have to use Lemma \ref{garside} b) : in the second case, $x_{i}|_{L}U(a+2,c)\ldots \beta$, hence $x_{a}|_{L}\beta$, in the third case, $x_{a}|_{L}x_{a+1}$(factors with index $\geq a+2$)$\beta$, and again $x_{a}|_{L}\beta$. Now suppose that $b\in\inn(\beta)$. If $D(a,b)$ is the first factor of $\gamma$, the same argument is correct, otherwise $\gamma=\ldots U(c,b-1)D(a,b)\ldots$ or $\gamma=\ldots D(b-1,c)D(a,b)\ldots$. We show that $x_{b}$ is a left divisor of $D(a,b)\ldots \beta$ and next repeat the same argument: Lemma \ref{garside} a) gives $x_{b}|_{L}x_{b-1}D(a,b)\ldots \beta$ and $x_{b}|_{L}D(b-1,c)D(a,b)\ldots \beta$ respectively and Lemma \ref{garside} b) gives $x_{b}|_{L}D(a,b)\ldots \beta$ and $x_{b}|_{L}D(b-2,c)D(a,b)\ldots \beta$ (in the last case) and finally $x_{b}|_{L}D(a,b)\ldots \beta$ in both cases.
\end{proof}
\begin{remark}If $\gamma=x_{2}x_{3}$, $\beta=x_{1}x_{2}x_{3}$ and $\delta=x_{1}x_{2}$, we have $\gamma\beta=\beta\delta$, $1\in\e-supp(\gamma)\cap \inn(\beta)$ but the intersection $\inn(\beta)\cap\inn(\gamma)$ is empty. This explains the long computations of the next Lemma.
\end{remark}
In the next statement and in the proof of Theorem \ref{canonical:conj sq fre} we will use the \emph{shift} of a word in $x_{1}, x_{2},\ldots $ given by  $x_{i}\longrightarrow x_{i+1}$ (for instance, if $w=x_{3}x_{2}x_{5}$, then $\Sigma ^{2}w=x_{5}x_{4}x_{7}$ and $\Sigma^{-1}w=x_{2}x_{1}x_{4}$).
\begin{lemma}(\textbf{Back and forth})  Consider two disjoint nondistant cycles $\gamma_{1}\prec \gamma_{2} $ with $\supp(\gamma_{1})=[b,c-1]$, $\supp(\gamma_{2})=[c+1,e]$ and $\beta$, $\delta$ two positive  braids. We have the following implications:\\
\indent a) If $\gamma_{1}\beta=\beta\delta$ and $x_{c}|_{L}\beta$, then there exists a positive $\beta'$ such that
$$\beta=D(c,b)\beta'\hbox{ and } \Sigma(\gamma_{1})\beta'=\beta'\delta;$$
\indent b) If $\gamma_{2}\beta=\beta\delta$ and $x_{c}|_{L}\beta$, then there exists a positive $\beta''$ such that
$$\beta=U(c,e)\beta''\hbox{ and } \Sigma^{-1}(\gamma_{2})\beta''=\beta''\delta;$$
\indent c) If $(\gamma_{1}\gamma_{2})\beta=\beta\delta$ and $x_{c}|_{L}\beta$, then there exists a positive $\beta'''$ such that
\begin{eqnarray*}
   \beta&=& D(c,b)D(c+1,b+1)\ldots D(e,e-c+b)\beta''' \\
   &=& U(c,e)U(c-1,e-1)\ldots U(b,e-c+b)\beta'''
\end{eqnarray*}
 and $\Sigma^{-c+b-1}(\gamma_{2})\Sigma^{e-c+1}(\gamma_{1})\beta'''=\beta'''\delta.$
\end{lemma}
\begin{proof}
 a) By induction on $k$ (from $c$ to $b$) we suppose that $\beta=D(c,k)$ is a left divisor of  $\gamma_{1}\beta$ and of $\beta$ and we have to show that $D(c,k-1)$ is also a left divisor of $\beta$. For an index $k$ in the interval $[b+1,c-1]$ we have $x_{k}|\gamma_{1}$, then the simple braid $\gamma_{1}$ has the form $\gamma_{1}=F_{1}(\leq k-2)x_{k-1}F_{2}(\geq k)$,  where $F(\leq m)$ and $F(\geq m)$ represent factors with supports having $m$ as the upper bound and the lower bound respectively. In the first case we have
 \begin{tabbing}
   0\indent \= 1\kill
   % \> for next tab, \\ for new line...
  \> $D(c,k)|_{L}F_{1}(\leq k-2)x_{k-1}F_{2}(\geq k)D(c,k)\beta_{0}=$ \\
  \> $=F_{1}(\leq k-2)x_{k-1}D(c,k)\Sigma(F_{2})\beta_{0}=D(c,k+1)F_{1}(\leq k-2)x_{k-1}x_{k}\Sigma (F_{2})\beta_{0}$
 \end{tabbing}
  and from $x_{k}|_{L}F_{1}x_{k-1}x_{k}\Sigma(F_{2})\beta_{0}$,  we obtain
  $x_{k-1}|_{L}\Sigma(F_{2})\beta_{0}$ (Garside Lemma \ref{garside})  and next $x_{k-1}|_{L}$ (Proposition \ref{leter:F}) and this ends the induction step in the first case. In the second case we have
   \begin{tabbing}
     0\indent\=1 \kill
     % \> for next tab, \\ for new line...
     $D(c,k)|_{L}F_{1}(\leq d-1)D(a,k-1)D(k-2,d)F_{2}(\geq d+1)D(c,k)\beta_{0}=$ \\
     \quad $=F_{1}(\leq d-1)D(a,k-1)D(c,k)D(k-2,d)\Sigma(F_{2})\beta_{0}=$ \\
     \quad $=D(c,a+2)F_{1}(\leq d-1)(x_{a}x_{a+1})(x_{a-1}x_{a})\ldots(x_{k-1}x_{k})D(k-2,d)F_{3}(\geq a+2)\beta_{0}$
   \end{tabbing}
   and from $x_{a+1}|_{L}F_{1}\cdot (x_{a}x_{a+1})\ldots(x_{k-1}x_{k})D(k-2,d)F_{3}\beta_{0}$ we obtain $x_{a}|_{L}(x_{a-1}x_{a})\ldots$\\$(x_{k-1}x_{k})D(k-2,d)F_{3}\beta_{0}$ (Lemma \ref{garside})  and a second induction (from $a$ to $k-1$) implies $x_{k}|_{L}(x_{k-1}x_{k})D(k-2,d)F_{3}\beta_{0}$, and  finally $x_{k-1}|_{L}D(k-2,d)F_{3}\beta_{0}$. Proposition \ref{leter:F} implies $x_{k-1}|_{L}\beta_{0}$, the end of the inductive step in this case. In the third case we have
\quad \quad \quad \quad  \begin{tabbing}
     0\=1 \kill
     % \> for next tab, \\ for new line...
    \> \quad \quad \quad \quad $D(c,k)|_{L}F_{1}(\leq d-1)(x_{k-1}\ldots x_{d})F_{2}(\geq k)D(c,k)\beta_{0}= $\\
    \> \quad \quad \quad \quad \quad $=F_{1}(\leq d-1)(x_{k-1}\ldots x_{d})D(c,k)\Sigma(F_{2})\beta_{0}=$ \\
    \> \quad \quad \quad \quad \quad $=D(c,k+1)F_{1}(\leq d-1) x_{k-1}x_{k}\Sigma(F_{2})(x_{k-2}\ldots x_{d})\beta_{0}$
   \end{tabbing}
  and from $x_{k}|_{L}F_{1}x_{k-1}x_{k}\Sigma(F_{2}(\geq k+1))(x_{k-2}\ldots x_{d})\beta_{0}$ we obtain $x_{k-1}|_{L}\Sigma (F_{2})(x_{k-2}\ldots$\\$ x_{d}\beta_{0})$ (Proposition \ref{leter:F} and Lemma \ref{garside}), next $x_{k-1}|_{L}(x_{k-2}\ldots x_{d})\beta_{0}$ (again Lemma \ref{garside})
 and the final step $x_{k-1}|_{L}\beta_{0}$ (Proposition \ref{leter:F}). The second equality of part a) is a consequence of
 $$D(c,b)\Sigma(\gamma_{1})\beta'=\gamma_{1}D(c,b)\beta'=D(c,b)\beta'\delta.$$

\indent  b) If $\gamma_{2}(x_{c}\beta_{1})=(x_{c}\beta_{1})\delta$, $\supp(\gamma_{2})=[c+1,e]$, then conjugation by Garside element $\Delta_{n}$ gives $\gamma_{3}(x_{m}\beta_{2})=(x_{m}\beta_{2})\delta'$, where $\supp(\gamma_{3})=[n,m-1]$; applying part a)  of the Lemma we obtain $\beta_{2}=D(m-1,n)\beta'_{2}$ and conjugating again by $\Delta_{n}$ we find $\beta_{1}=U(c+1,e)\beta''$. Now $U(c,e)\Sigma^{-1}(\gamma_{2})\beta''=\gamma_{2}U(c,e)\beta''=U(c,e)\beta''\delta$ implies the second equation in part b).\\
\indent  c) We will use the first two parts  in the form given in the proof: \\
\noindent a$'$) if $D(c,k)|_{L}\gamma_{1}\beta,\,D(c,k)|_{L}\beta$, then $D(c,k+1)|_{L}\beta$, for $k\geq b+1$;\\
\noindent b$'$) if $U(c,m)|_{L}\gamma_{2}\beta,\,\,U(c,m)|_{L}\beta$, then $U(c,m+1)|_{L}\beta$, for $m\leq e-1$\\
\noindent ( part b$'$) is equivalent to a$'$) after a conjugation with  Garside braid). From $x_{c}|_{L}\gamma_{1}\gamma_{2}\beta$
 we infer $x_{c}|_{L}\gamma_{2}\beta$ (Proposition \ref{leter:F}) and $\beta=U(c,e)\beta_{0}$ (part b$'$). By induction we suppose that  $\beta=U(c,e) U(c-1,e-1)\ldots U(c-j,e-j)\beta_{j}$. From hypothesis
 $$x_{c}|_{L}\gamma_{1}\gamma_{2}U(c,e)\ldots U(c-j,e-j)\beta_{j}=\gamma_{1}U(c,e)\ldots U(c-j,e-j)\Sigma^{-j-1}(\gamma_{2})\beta_{j}$$
and also $x_{c}|_{L}U(c,e)\ldots U(c-j,e-j)\Sigma^{-j-1}(\gamma_{2})\beta_{j})$ and part a$'$) implies  $D(c,b)|_{L}U(c,e)$\\$U(c-1,e-1)\ldots U(c-j,e-j)\Sigma^{-j-1}(\gamma_{2})\beta_{j})$, hence $D(c-1,b)|_{L}U(c+1,e)U(c-1,e-1)\ldots U(c-j,e-j)\Sigma^{-j-1}(\gamma_{2})\beta_{j}$.
Garside Lemma \ref{garside} implies $D(c-1,b)|_{L}U(c-1,e-1)\ldots U(c-j,e-j)\Sigma^{-j-1}(\gamma_{2})\beta_{j}$ hence $D(c-2,b)|_{L}U(c,e-1)\ldots U(c-j,e-j)\Sigma^{-j-1}(\gamma_{2})\beta_{j}$. A second induction gives $D(c-j,b)|_{L} U(c-j,e-j)\Sigma^{-j-1}(\gamma_{2})\beta_{j}$ therefore $D(c-j-1,b)|_{L}U(c-j+1,e-j)\Sigma^{-j-1}(\gamma_{2})\beta_{j}$ and finally $D(c-j-1,b)|_{L}\Sigma^{-j-1}(\gamma_{2})\beta_{j}$. After $j+1$ desuspensions $\supp(\gamma_{2})=[c+1,e]$ becomes $[c-j,e-j-1]$, so we can use Proposition \ref{leter:F} to obtain $x_{c-j-1}|_{L}\beta_{j}$ and again part b$'$) for $U(c-j-1,e-j-1)|_{L}\beta_{j}$ and this complete the first half of part c). The last equality of part c) is a consequence of the relation
  $$D(c,b)D(c+1,b+1)\ldots D(e,e-c+b)=U(c,e)U(c-1,e-1)\ldots U(b,e-c+b),$$
  (start an induction by the length of $D(c,b)$ with the equality $D(c,b)D(c+1,b+1)\ldots D(e,e-c+b)=U(c,e)D(c-1,b)D(c,b+1)\ldots D(e-1,e-c+b)$):
  \begin{tabbing}
    0\indent\indent\indent \= 1 \kill
    % \> for next tab, \\ for new line...
    $(\gamma_{1}\gamma_{2})\beta$ \>$= (\gamma_{1}\gamma_{2})D(c,b)D(c+1,b+1)\ldots D(e,e-c+b)\beta'''$ \\
     \>$ =(\gamma_{1}\gamma_{2})U(c,e)U(c-1,e-1)\ldots U(b,e-c+b)\beta'''$ \\
     \>$ =\gamma_{1}U()\Sigma^{-1}(\gamma_{2})U(c-1,e-1)\ldots U(b,e-c+b)\beta'''$ \\
     \> $=\gamma_{1}U(c-1,e-1)U(c-1,e-1)\ldots U(b,e-c+b)\Sigma^{-c+b-1}(\gamma_{2})\beta'''$ \\
     \> $=\gamma_{1}D(c,b)D(c+1,b+1)\ldots D(e,e-c+b)\Sigma^{-c+b-1} (\gamma_{2})\beta'''$\\
     \> $=D(c,b)\Sigma(\gamma_{1})D(c+1,b+1)\ldots D(e,e-c+b)\Sigma^{-c+b-1} (\gamma_{2})\beta'''$ \\
     \> $=D(c,b)D(c+1,b+1)\ldots D(e,e-c+b)\Sigma^{e-c+1}(\gamma_{1})\Sigma^{-c+b-1} (\gamma_{2})\beta'''$
  \end{tabbing}
and this is equal to $D(c,b)D(c+1,b+1)\ldots D(e,e-c+b)\beta'''\delta$ by hypothesis. The final remark is that $\supp(\Sigma^{e-c+1}(\gamma_{1}))=[e-c+b+1,e]$ and $\supp (\Sigma^{-c+b-1}(\gamma_{2}))=[b,e-c+b-1]$, hence the two suspensions commute, but essential for the next proof is the fact that $\gamma_{3}=\Sigma^{-c+b-1}(\gamma_{2})\Sigma^{e-c+1}(\gamma_{1})$ is also literally simple.
\end{proof}

\begin{proposition}\label{a and b}Suppose that $\alpha\in \mathcal{LSB}_{n}$ and $\beta,\delta \in\mathcal{MB}_{n}$:\\
\indent a) $\alpha\beta=\beta\delta$ implies that $\delta\in \mathcal{\mathcal{LSB}}_{n}$;\\
\indent b) $\beta\alpha=\delta\beta$ implies that $\delta\in \mathcal{\mathcal{LSB}}_{n}$.
\end{proposition}

\begin{remark}a) The two parts of Proposition \ref{a and b} are equivalent: $\beta\alpha=\delta\beta$  implies $\Rev(\alpha)\Rev(\beta)=\Rev(\beta\alpha)=\Rev(\delta\beta)=\Rev(\beta)\Rev(\delta)$ with $\Rev(\alpha)$ literally simple; from a) we obtain $\Rev(\delta)\in\mathcal{LSB}_{n}$, hence $\delta\in\mathcal{LSB}_{n}$.\\
\indent b) If $\alpha,\alpha' \in \mathcal{LSB}_{n}$, are conjugate and Proposition \ref{a and b} a) is true for $\alpha$, then it is true for $\alpha'$ too: if $\alpha'\beta=\beta\delta$ and $\alpha\gamma=\gamma\alpha'$, then $\alpha(\gamma\beta)=\gamma\alpha'\beta=(\gamma\beta)\delta$ and Proposition \ref{a and b} a) for $\alpha$ implies $\delta\in\mathcal{LSB}_{n}$. If $\varepsilon\alpha=\alpha'\varepsilon$, then $\varepsilon\alpha\varepsilon^{-1}\beta=\beta\delta$, hence $\alpha(\varepsilon^{-1}\beta)=(\varepsilon^{-1}\beta)\delta$. Multiplying both sides with a big power $\Delta^{2k}$ we obtain a positive braid $\beta'=\varepsilon^{-1}\beta\Delta^{2k}$ and $\alpha\beta'=\beta'\delta$, therefore $\delta\in\mathcal{LSB}_{n}$.
\end{remark}

\emph{Proof of Proposition \ref{a and b}} \emph{a)}  We use a double induction on the length of $\alpha$ and on the length of $\beta$. If $|\alpha|\leq 1$, then $\alpha\beta=\beta\delta$ implies $|\delta|\leq 1$, so $\delta\in\mathcal{LSB}_{n}$. Now we start induction on $|\beta|$ (the case $|\beta|=0$ is obvious). We will discuss three cases, the first trivial, the second a simple consequence of gud Lemma, the third a consequence of baf Lemma. We put $\alpha=\gamma_{1}\gamma_{2}\ldots\gamma_{s}$ (as an increasing product of disjoint cycles).\\
\indent \emph{Case 1:} there is index $k\in\inn(\beta)\setminus \bigcup^{s}\limits_{i=1}\e-supp(\gamma_{i})$. In this case $x_{k}$ commutes with all $\gamma_{i}$ and $\beta=x_{k}\beta'$: hypothesis $x_{k}\alpha\beta'=\alpha(x_{k}\beta')=(x_{k}\beta')\delta$ implies $\alpha\beta'=\beta'\delta$, $|\beta'|<|\beta|$ and inductive step gives $\delta \in \mathcal{LSB}_{n}$.\\
\indent \emph{Case 2:} there is an index $k\in \inn(\beta)\bigcap\supp(\alpha)$. In this case, using gud Lemma, one can find an index $j\in\inn(\beta)\bigcap\inn(\gamma_{i})$, $\gamma_{i}=x_{j}\gamma'_{i}$ with $\gamma'_{i}\in\mathcal{LSB}_{n}$, $\beta=x_{j}\beta'$;
$$x_{j}(\gamma_{1}\ldots\gamma'_{i}\ldots \gamma_{s})(x_{j}\beta')=(\gamma_{1}\ldots\gamma_{i}\ldots\gamma_{s})(x_{j}\beta')=(x_{j}\beta')\delta$$ implies $(\gamma_{1}\ldots\gamma'_{i}\ldots\gamma_{s}x_{j})\beta'=\beta'\delta$. The new braid $\alpha'=(\gamma_{1}\ldots\gamma'_{i}\ldots\gamma_{s}x_{j})$ is literally simple ($x_{j}$ was deleted from some place  in $\gamma_{i}$ next added, at another place), $|\alpha'|=|\alpha|$,  $|\beta'|<|\beta|$, and again inductive step gives $\delta\in \mathcal{LSB}_{n}$.

\emph{Case 3:} there is an index $k$ in $\inn(\alpha)$ and also on the boundary $\bigcup^{s}\limits_{i=1}[\e-supp(\gamma_{i})\setminus\supp(\gamma_{i})]$. We have three subcases:

3.1) there is an index $i$ such that $\supp(\gamma_{i})=[b,c-1]$ and $\gamma_{i+1}$ is distant from $\gamma_{i}$ (or simply $i=s$). We can apply baf Lemma a) because $x_{b},\ldots x_{c-1},x_{c}$ commute with factors $\gamma_{j}$, $j\neq i$ and we obtain from
$\alpha\beta=(\gamma_{1}\ldots\gamma_{i}\ldots \gamma_{s})D(c,b)\beta'$$=D(c,b)\beta'\delta$ the equality $(\gamma_{1}\ldots\Sigma(\gamma_{i})\ldots \gamma_{s})\beta'=\beta'\delta$ with $\gamma_{1}\ldots\Sigma(\gamma_{i})\ldots \gamma_{s}$ literally simple and $|\beta'|<|\beta|$.

3.2) there is an index $i+1$ such that $\supp(\gamma_{i+1})=[c+1,e]$ and $\gamma_{i}$ is distant from
 $\gamma_{i+1}$ (or $i+1=1$). Baf Lemma b) gives, as in previous case $(\gamma_{1}\ldots\Sigma^{-1}(\gamma_{i+1})\ldots \gamma_{s})\beta''=\beta''\delta$ where $|\beta''|<|\beta|$ and $(\gamma_{1}\ldots\Sigma^{-1}(\gamma_{i+1})\ldots \gamma_{s})\in \mathcal{LSB}_{n}$.

3.3) there is an index $i$ such that $\supp(\gamma_{i})=[b,c-1]$, $\supp(\gamma_{i+1})=[c+1,e]$. The third part of baf Lemma implies $(\gamma_{1}\ldots\Sigma^{-c+b-1}(\gamma_{i+1})\Sigma^{e-c+1}(\gamma_{i})\ldots \gamma_{s})\beta'''=\beta'''\delta$, where the length of $\beta'''$ is smaller than $|\beta|.$\begin{flushright} $\Box$ \end{flushright}
\begin{cor}\label{l=c}$\mathcal{LSB}_{n}=\mathcal{CSB}_{n}$.
\end{cor}
\section{\textbf{The invariant simple set}}

If $(A_{i})_{i\in I}\subseteq \mathcal{MB}_{n}$  are invariant under conjugation, then $\bigcup\limits_{i\in I}A_{i}$ is also invariant under conjugation  and this explains the definition of $\mathcal{ISB}_{n}$.
The definition of conjugate simple braids implies  the inclusion $\mathcal{CSB}_{n}\subseteq \mathcal{ISB}_{n}$. The reverse inclusion  is also a direct consequence of this definition:
\begin{lemma}If $\alpha \in \Div(\Delta_{n})\setminus \mathcal{LSB}_{n}$, then there are positive braids  $\beta\in\mathcal{MB}_{n}$ and $\alpha' \in \mathcal{MB}_{n}\setminus \Div(\Delta_{n})$ such that $\alpha\beta=\beta\alpha'$.
\end{lemma}
\begin{proof}If $\alpha$ is not in $\mathcal{LSB}_{n}$, then $\alpha$ is not in $\mathcal{CSB}_{n}$, hence there is  a conjugate $\alpha'=\beta^{-1}\alpha\beta\in\mathcal{MB}_{n}\setminus\Div(\Delta_{n})$, and $\beta$ can be chosen to be positive.
\end{proof}
\begin{cor}\label{l=c=i}$\mathcal{LSB}_{n}=\mathcal{CSB}_{n}=\mathcal{ISB}_{n}$.
\end{cor}

Now we find the smallest positive braid of a conjugacy class containing (literally) simple braids.

\emph{Proof of Theorem \ref{canonical:conj sq fre}}  \emph{a)} From Lemma (\ref{crit:conj sq fre}), $\beta_{k_{1},j_{1}}$ commutes with $\beta_{k_{i},j_{i}}$ for $i\geq 3$.  If $j_{2}> k_{1}+1$, we can write $\beta_{K,J}=\beta_{k_{2},j_{2}}\beta_{k_{1},j_{1}}\ldots
\beta_{k_{s},j_{s}}$ (the same number of blocks). If $j_{2}=
k_{1}+1$, then conjugating with $\beta_{k_{1},j_{1}}$, we have
 $$\beta_{K,J}\thicksim (\beta_{k_{2},j_{2}}\beta_{k_{1},j_{1}})\ldots
\beta_{k_{s},j_{s}}= \beta_{k_{2},j_{1}}\ldots
\beta_{k_{s},j_{s}}\quad(\hbox{one}\,\beta\,\hbox{block less}).$$ Now repeat the process for the
pair $j_{i}$ and $k_{i-1}$. Finally we have $\beta_{K,J}
\sim\beta_{C,D} =\beta_{c_{1},d_{1}} \beta_{c_{2},d_{2}}\ldots
\beta_{c_{r},d_{r}}$, where $r\leq s,\,d_{i}\geq c_{i+1}+2$ and
$C_{\star},\,D_{\star}$ are decreasing
sequences. Now we conjugate $\beta_{C,D}$ in order to obtain a similar $\beta_{E,F}$ satisfying the same conditions
and also all differences $f_{i}-e_{i+1}$ are equal to 2 and the first
index $e_{1}$ is $n-1.$ If in $\beta_{C,D}$ we have a difference $d_{i-1}-c_{i}\geq 3$ or
the first letter is not $n-1$, then we can shift one step the block
$\beta_{c,d}=\beta_{c_{i},d_{i}}$ by conjugating with
$\beta_{c+1,d}$:
$$\beta_{c+1,d}\beta_{c+1,d+1}=\beta_{c,d}\beta_{c+1,d}$$
 Continue in this way until we have all differences
equal to 2. Taking conjugate with $\Delta_{n}:\, x_{n-i}\Delta_{n}=\Delta_{n}x_{i}$ we obtain
$\beta_{A}$ (but $A$ is not necessary in decreasing order).
If we have two consecutive blocks $\beta_{a,a+l}\beta_{b,b+m}$ and $m>l$ ( by the last step we have $b= a+l+2$), turn it into $\beta_{a,a+m}\beta_{b+m-l,b+m}$ by conjugating with appropriate shifts of $\Delta$.
\begin{eqnarray*}
   (\Sigma^{a-1}\Delta_{b-a+m+2})\beta_{a,a+l}\beta_{b,b+m}(\Sigma^{a-1}\Delta_{b-a+m+2})^{-1}&=& \beta_{b+m,b+m-l}\beta_{a+m,a}\\
   &=&\beta_{a+m,a}\beta_{b+m,b+m-l}.
\end{eqnarray*}

Now we conjugate separately the two blocks to put them in increasing order:

  $(\Sigma^{a-1}\Delta_{m+1}) (\Sigma^{b+m-l-1}\Delta_{l+1})\beta_{a+m,a}\beta_{b+m,b+m-l}(\Sigma^{b+m-l-1}\Delta_{l+1})^{-1} (\Sigma^{a-1}\Delta_{m+1})^{-1} $\\

    $=[(\Sigma^{a-1}\Delta_{m+1})\beta_{a+m,a} (\Sigma^{a-1}\Delta_{m+1})^{-1}] [(\Sigma^{b+m-l-1}\Delta_{l+1})\beta_{b+m,b+m-l} (\Sigma^{b+m-l-1}\Delta_{l+1})^{-1}]$

      $= \beta_{a,a+m}\beta_{b+m-l,b+m}.$

\emph{b)} If $\beta_{A}\sim\beta_{A'}$ then $\widehat{\beta}_{A}$ is
equivalent to $\widehat{\beta}_{A'}$ as links in a solid torus
$\mathbb{T}$. Let us denote $s_{h}=a_{1}+a_{2}+\ldots+a_{h}$ ($s_{0}=0$).  The
link $\widehat{\beta}_{A}$ has $r$ components given by the
$r$-blocks $\{\widehat{x_{s_{j-1}+1}\ldots x_{s_{j}}}\}_{j=1,\ldots,r}$   plus $n-s_{r}$
  components given by trivial strands $s_{r}+3,\ldots,n.$  The trivial components of $\widehat{\beta}_{A}$ give
the generator  $t$ of $H_{1}(\mathbb{T})$ and the non trivial
 components $\widehat{x_{s_{j-1}+1}\ldots x_{s_{j}-1}}$  give the cycle
 $a_{j}t$ in $H_{1}(\mathbb{T})$. The homology classes of the link components are isotopy invariants of link in the solid
 torus and the proof is finished. \begin{flushright} $\Box$ \end{flushright}
\emph{Proof of Corollary \ref{cor:diagram}} Remark that the natural section $s:\,\Sigma_{n}\longrightarrow\Div(\Delta_{n})$ is a bijective partial group homomorphism: if $\alpha,\beta\in\Sigma_{n}$ have images satisfying $s(\alpha)s(\beta)\in\Div(\Delta_{n})$, then $s(\alpha\beta)=s(\alpha)s(\beta)$.\\
\indent Theorem \ref{canonical:conj sq fre} gives canonical forms for conjugacy classes of simple braids and these are in bijection (induced by $\pi$) with conjugacy classes of the symmetric group.
\begin{flushright}
$\Box$
\end{flushright}
\section{\textbf{Markov simple braids}}
\begin{lemma} If $\gamma_{1}\prec\gamma_{2}\prec \ldots \prec\gamma_{r}$ are disjoint cycles, then $\beta=\gamma_{1}\gamma_{2}\ldots\gamma_{r}$ is a Markov simple braid.
\end{lemma}
\begin{proof}The braid $\beta$ is literally simple=conjugate simple; in a Markov chain $\beta=\beta_{1}\rightarrow\beta_{2}\rightarrow\ldots\rightarrow\beta_{s}=\beta'$ a move $MI$ $\beta_{i}\rightarrow\beta_{i+1}$ transforms a conjugate simple braid into a conjugate simple braid and a move $MII_{+}$ transforms a literally simple braid into a literally simple braid (and also a change in the diagram of a literally simple braid preserves simplicity).
\end{proof}
\begin{lemma}If $\beta$ is Markov simple braid, then $\beta\in \mathcal{LSB}_{n}$.
\end{lemma}
\begin{proof}If $\beta$ is Markov simple then $\beta$ is conjugate simple=literally simple.
\end{proof}
\begin{cor}\label{l=c=i=m}$\mathcal{LSB}_{n}=\mathcal{CSB}_{n}=\mathcal{ISB}_{n}=\mathcal{MSB}_{n}$.
\end{cor}
\section{\textbf{Geometrically simple braids}}
Canonical form of the conjugacy classes in Theorem \ref{canonical:conj sq fre}  shows that $\mathcal{CSB}_{n}\subseteq \mathcal{GSB}_{n}$.

\begin{lemma}\label{diagram}If the closure $\widehat{\beta}$ of the positive $n$-braid $\beta$ is a trivial $c$-link ($c\geq 2$ components), then the diagram of $\widehat{\beta}$ has  $c$ seperated components.
\end{lemma}
\begin{proof}Let us suppose that in the diagram of the closure of the braid $\beta$ there are two non separated components, $C_{1}$, $C_{2}$; this implies that there are crossings between $C_{1}$ and $C_{2}$, and these crossings should be in the braid diagram (the threads added to close the braid have no crossing). The braid $\beta$ is positive, hence every crossing has a $-\frac{1}{2}$ contribution to the linking number $lk(C_{1},C_{2})$, but this is zero.
\end{proof}
In~\cite{Barbu:Rehana} a Laurent polynomial invariant of oriented links $D$ is introduced, a new specialization of HOMFLY polynomial: $(l,m)\mapsto(s,-2)$, with skein relation $$sD(L_{+})+s^{-1}D(L_{-})-2D(L_{0})=0$$
and expansion formula of the closure of the $n$-braid $\beta=x_{i_{1}}^{a_{1}}\ldots x_{i_{k}}^{a_{k}}$ given by \\ $D_{n}(.\,.,a_{j},.\,.)(s)=(1-a_{j})s^{a_{j}}D_{n}(.\,.,a_{j-1},0,a_{j+1},.\,.)+a_{j}s^{a_{j}-1}D_{n}(.\,.,a_{j-1},1,a_{j+1},.\,.).$
\begin{proposition}\label{Geo:trivial}Suppose that $\beta$ is a positive braid, $\beta\in\mathcal{MB}_{n}$ with a maximal support $\supp(\beta)=[1,n-1]$.\\
a) If $\deg(\beta)=n-1$, then $D_{n}(\beta)=1$:\;\\
b) If $\deg(\beta)\geq n$, then $D_{n}(\beta)$ is a polynomial in $s$ and $0$ is one of its roots.
\end{proposition}
\begin{proof} a) The first part is a consequence of the following facts:\\
\indent\emph{a1)} $\beta$ is a literally simple braid;\\
\indent\emph{a2)} $\beta$ is conjugate to $x_{1}\ldots x_{n-1}$ (Theorem \ref{canonical:conj sq fre});\\
\indent\emph{a3)} $D_{n}(\beta)=D_{n}(x_{1}\ldots x_{n-1})=D(\bigcirc)=1$ (see \cite {Barbu:Rehana} Corollary 5.6 for a general formula).\\
\indent \emph{b)} The second part is proved by a triple induction; on $n$, on the factor length $k$ (the number of distinct factors of $\beta$), and on $\deg(\beta)$.\\
\indent In $\mathcal{MB}_{2}$, $D_{2}(x_{1}^{n})(s)=\frac{1}{2}[(1-n)s^{n+1}+(1+n)s^{n-1}]$ (see \cite{Barbu:Rehana}, Example 4.3), therefore for $n=2$ the claim is true. Now consider a positive braid $\beta=x_{i_{1}}^{a_{1}}\ldots x_{i_{k}}^{a_{k}}\in \mathcal{MB}_{n}$, all exponents are $\geq 1$ (and $i_{h}\neq i_{h+1}$). The support of $\beta$ contains all indices and $\deg(\beta)=\sum^{k}\limits_{i=1}a_{i}\geq n$, therefore $k\geq n-1$.\\
\indent We want to prove the claim for $k=n-1$. After a conjugation (cyclic permutation of factors) we can suppose that $\beta=x_{i_{1}}^{a_{1}}\ldots x_{i_{n-2}}^{a_{n-2}}x_{n-1}^{a_{n-1}}=\beta_{0}x_{n-1}^{a}$ with $\beta_{0}\in\mathcal{MB}_{n-1}$ and $\supp(\beta_{0})=[1,n-2]$. If $a=1$, then $\beta=\beta_{0}x_{n-1},\,\deg(\beta_{0})\geq n-1$, and also $D_{n}(\beta)=D_{n-1}(\beta_{0})$; induction on $n$ shows that $D_{n-1}(\beta_{0})$ is a polynomial in $s$ and $D_{n-1}(\beta_{0})(0)=0$. If $a\geq 2$, the expansion formula (in the last position) gives
\begin{eqnarray*}
  D_{n}(\beta)(s) &=& (1-a)s^{a}D_{n}(\beta_{0})+a s^{a-1}D_{n}(\beta_{0}x_{n-1})\\
   &=& (1-a)s^{a}\frac{s^{2}+1}{2s}D_{n-1}(\beta_{0})+a s^{a-1}D_{n-1}(\beta_{0}),
\end{eqnarray*}
where $D_{n-1}(\beta_{0})$ is a polynomial (possibly constant=1), therefore $D_{n}(\beta)$ is also a polynomial without constant term.

Now suppose $k\geq n$. If one of the exponents $a_{j}$ is $\geq 2$, we reduce the degree:
$$D_{n}(\beta)=D_{n}(\beta x_{i_{j}}^{a_{j}}\beta_{2})=(1-a_{j})s^{a_{j}}D_{n}(\beta_{1}\beta_{2})+a_{j}s^{a_{j}-1}D_{n}(\beta_{1}x_{i_{j}}\beta_{2})$$
\indent If $\supp(\beta_{1}\beta_{2})=[1,n-1]$, then $D_{n}(\beta_{1}\beta_{2})$ and $D_{n}(\beta_{1}x_{i_{j}}\beta_{2})$ are polynomials  and $D_{n}(\beta)(0)=0$. Suppose that $\supp(\beta_{1}\beta_{2})\subsetneq [1,n-1]$. If $i_{j}=n-1$ (or 1), then $D_{n}(\beta_{1}\beta_{2})=\frac{s^{2}+1}{2s}D_{n-1}(\beta_{1}\beta_{2})$ and $\supp(\beta_{1}\beta_{2})=[1,n-2]$ (in the case $i_{j}=1$, after a conjugation with Garside braid), and again $D_{n}(\beta)$ is a polynomial with zero constant term. In the case $i=i_{j}\in\{2,3,\ldots,n-2\}$, $\widehat{\beta_{1}\beta_{2}}$ has two separated components, each of them are closures of positive braids $\gamma_{1}\in\mathcal{MB}_{i}$ and $\gamma_{2}\in\Sigma ^{i-1}\mathcal{MB}_{n-i}$ respectively, with $\supp(\gamma_{1})=[1,i-1]$, $\supp(\Sigma ^{-i
+1}\gamma_{2})=[1,n-i-1]$, and $D_{n}(\beta_{1}\beta_{2})=\frac{s^{2}+1}{2s}D_{i}(\gamma_{1})D_{n-i}(\gamma_{2})$ . The second term $D_{n}(\beta_{1}x_{i}\beta_{2})$ is a polynomial (possibly 1) because $\supp(\beta_{1}x_{i_{j}}\beta_{2})=[1,n-1]$, therefore  in this case also $D_{n}(\beta)$ is a polynomial in $s$, equal to $0$ for $s=0$.

The last case is when all the exponents $a_{i}=1$. As degree of $\beta$ is $\geq n$, $\beta$ cannot be literally simple, therefore $\beta$ has a (positive) conjugate $\beta'$ containing exponents $\geq 2$; because $\supp(\beta')=\supp(\beta)=[1,n-1]$, factor length $(\beta')<$ factor length$(\beta)$, the inductive hypothesis (on $k$) implies the result.
\end{proof}

\begin{lemma}\label{c:b}If the closure $\widehat{\beta}$ of the positive $n$-braid $\beta$ is a trivial knot, then $\beta$ is literally simple.
\end{lemma}
\begin{proof}If $\widehat{\beta}$ is a knot, the support of $\beta$ should be maximal: $\supp(\beta)=[1,n-1]$. If $\widehat{\beta}$ is a trivial knot, $D_{n}(\beta)=D(\widehat{\beta})=1$ and Proposition \ref{Geo:trivial} implies $\deg(\beta)=n-1$, therefore $\beta$ is literally simple.
\end{proof}
\emph{Proof of Theorem \ref{main: th}}\indent From Corollary \ref{l=c}, \ref{l=c=i} and \ref{l=c=i=m}, it is enough to show $\mathcal{LSB}_{n}=\mathcal{CSB}_{n}=\mathcal{GSB}_{n}$.
If $\beta$ is geometrically simple braid, Lemma  \ref{diagram} implies that $\beta=\beta_{1}\beta_{2}\ldots \beta_{c}$ with disjoint supports  and any two of $\supp(\beta_{i})$ not consecutive. Each closure $\widehat{\beta_{i}}$ is a trivial knot and Lemma \ref{c:b} implies that each $\beta_{i}$ is literally simple, therefore $\beta$ is literally simple.
\begin{flushright}
$\Box$
\end{flushright}

\section{\textbf{Appendix}}
In this section we consider only positive braids: we compute the left least  common multiple $(l.c.m_{L})$ of a generator $x_{i}$ and of the very simple braid, $U(a,b)$ and $D(c,d)$ respectively. The simplest case appears in Garside~\cite{Garside:69}:
\begin{lemma}\label{garside}(\textbf{Garside}) Suppose that $x_{i},x_{j}\in\Div_{L}(\beta)$:\\
a) if $|i-j|\geq 2$, then $x_{i}x_{j}=x_{j}x_{i}|_{L}\beta$;\\
b) if $i+1=j$, then $x_{i}x_{i+1}x_{i}=x_{i+1}x_{i}x_{i+1}|_{L}\beta.$
\end{lemma}
\begin{lemma}\label{extended garside}a) If $x_{i}x_{i+1},x_{i+2}\in \Div_{L}(\beta)$, then $x_{i}x_{i+1}(x_{i+2}x_{i+1})=x_{i+2}(x_{i}x_{i+1}x_{i+2})|_{L}\beta;$\\
b) if $x_{i+1}x_{i+2},x_{i}\in\Div_{L}(\beta)$, then $x_{i}(x_{i+1}x_{i}x_{i+2}x_{i+1})=x_{i+1}x_{i+2}(x_{i}x_{i+1}x_{i+2})|_{L}\beta$;\\
c) if $x_{i+2}x_{i+1},x_{i}\in\Div_{L}(\beta)$, then $x_{i}(x_{i+2}x_{i+1}x_{i})=x_{i+2}x_{i+1}x_{i}(x_{i+1})|_{L}\beta$;\\
d) if $x_{i+1}x_{i},x_{i+2}\in\Div_{L}\beta$, then $x_{i+1}x_{i}(x_{i+2}x_{i+1}x_{i})=x_{i+2}(x_{i+1}x_{i}x_{i+2}x_{i+1})|_{L}\beta.$
\end{lemma}
\begin{proof}Case  a): Garside Lemma a) implies that $x_{i}x_{i+1}\beta'=\beta=x_{i}x_{i+2}\beta^{''}$, therefore $x_{i+2}|_{L}x_{i+1}\beta'$, and the case b) of the Lemma implies that $\beta=x_{i}(x_{i+1}x_{i+2}x_{i+1})\beta^{'''}$.\\
\indent Case b): Garside Lemma b) implies $x_{i+1}x_{i+2}\beta'=\beta=x_{i+1}x_{i}x_{i+1}\beta^{'''}$, therefore $x_{i}x_{i+1}$ and $x_{i+2}$ are left divisors of $x_{i+2}\beta'$; case a) of this Lemma gives the result.\\
\indent Case c) and d) can be checked in a similar way.
\end{proof}
Using Lemma \ref{garside} and Lemma \ref{extended garside} one can start  an induction to prove the next results (or one can find a proof in \cite{Barbu:usman}):
\begin{proposition}\label{l.c.m ,U}Suppose that $x_{i},U(a,b)\in \Div_{L}(\beta)$ ($a+1\leq b$). We have the following implications:\\
a) if $i\notin\e-supp U(a,b)$, then $x_{i}U(a,b)=U(a,b)x_{i}|_{L}\beta$;\\
b) if $i=a-1$, then $x_{a-1}D(a,a-1)D(a+1,a)\ldots D(b,b-1)=U(a,b)U(a-1,b)|_{L}\beta$;\\
c) if $i=a$, then l.c.m$_{L}(x_{a},U(a,b))=U(a,b)$;\\
d) if $i\in[a+1,b]$, then $U(a,b)x_{i-1}=x_{i}U(a,b)|_{L}\beta$;\\
e) if $i=b+1$, then $U(a,b)D(b+1,b)=x_{b+1}U(a,b+1)|_{L}\beta.$
\end{proposition}
\begin{proposition}\label{l.c.m,D}Suppose that $x_{i},D(c,d)\in \Div_{L}(\beta)\,(c\geq d+1)$. We have the following implications:\\
a) if $i\notin \e-suppD(c,d)$ then $x_{i}D(c,d)=D(c,d)x_{i}|_{L}\beta;$\\
b) if $i=d-1$, then $x_{d-1}D(c,d-1)=D(c,d)U(d-1,d)|_{L}\beta;$\\
c) if $i\in[d,c-1]$, then $x_{i}D(c,d)=D(c,d)x_{i+1}|_{L}\beta$;\\
d) if $i=c$, then l.c.m$_{L}(x_{c},D(c,d))=D(c,d)$;\\
e) if $i=c+1$, then $D(c,d)D(c+1,d)=x_{c+1}D(c,d)D(c+1,d+1)|_{L}\beta.$
\end{proposition}
\begin{proposition}\label{leter:F}Given $\beta\in\mathcal{MB}_{n}$ and a cycle $\gamma$, $\supp(\gamma)=[b,e]$, we have the following implications:\\
\indent a) if $x_{b-1}|_{L}\gamma\beta$, then $x_{b-1}|_{L}\beta;$\\
\indent b)  if $x_{e+1}|_{L}\gamma\beta$, then $x_{e+1}|_{L}\beta.$
\end{proposition}
\begin{proof}Induction on the length of $\gamma$ and Garside Lemma \ref{garside} give the result.
\end{proof}

\end{document}